\documentclass[12pt]{amsart}
\usepackage{fullpage}
\usepackage{amsfonts}
\usepackage{latexsym}
\usepackage{amscd}
\usepackage{amsmath}
\usepackage{amssymb}
\usepackage{enumerate}
\usepackage{graphicx}
\usepackage[mathscr]{eucal}
\usepackage{xypic}

\newtheorem{theorem}{Theorem}[section]
\newtheorem{lemma}[theorem]{Lemma}

\newtheorem{proposition}[theorem]{Proposition}

\theoremstyle{definition}
\newtheorem{definition}[theorem]{Definition}

\newtheorem{question}[theorem]{Question}

\newcommand{\co}{\mathfrak{c}}

\theoremstyle{remark}

\numberwithin{equation}{section}

\begin{document}

\title[Finite powers of selectively  pseudocompact groups]{Finite powers of selectively  pseudocompact groups}

\author[S. Garcia-Ferreira]{S. Garcia-Ferreira}

\address{Centro de Ciencias Matem\'aticas, Universidad Nacional Aut\'onoma de M\'exico, Campus Morelia, Apartado Postal 61-3, Santa Maria, 58089, Morelia, Michoac\'an, M\'exico}
\email{sgarcia@matmor.unam.mx}

\author[A. H. Tomita]{A. H. Tomita}
\address{Instituto de Matem\'atica e Estat\'istica, Universidade de S\~ao Paulo \\ Rua do Mat\~ao, 1010, CEP 05508-090, S\~ao Paulo, Brazil}
\email{tomita@ime.usp.br}

\thanks{Research of the first-named author was supported
by  CONACYT grant no. 176202 and PAPIIT grant no. IN105614. The second author  has support from CNPq (Brazil) - ``Bolsa de Produtividade em Pesquisa, processo 307130/2013-4''. The research leading to this paper was essentially performed while the first-listed author was visiting the Instituto de Matem\' atica and Estat\' \i stica of University of S\~ao Paulo. He would like to gratefully acknowledge the hospitality received from this institution and the financial support received from the project  ''Grupos topol´\' ogicos fortemente pseudocompactos '', Proc. FAPESP 2016/23440-4 (Aux\' \i lio professor visitante).}




\begin{abstract} A space $X$ is called {\it selectively  pseudocompact} if for each sequence $(U_{n})_{n\in \mathbb{N}}$ of pairwise disjoint nonempty open subsets of $X$ there is a sequence $(x_{n})_{n\in \mathbb{N}}$ of points in $X$ such that  $cl_X(\{x_n : n < \omega\}) \setminus \big(\bigcup_{n < \omega}U_n \big) \neq \emptyset$ and $x_{n}\in U_{n}$, for each $n < \omega$. Countably compact space spaces are  selectively  pseudocompact and every selectively  pseudocompact space  is pseudocompact. We show, under the assumption of $CH$, that for every positive integer $k > 2$ there exists a topological group  whose $k$-th power is countably compact but its $(k+1)$-st power  is not selectively  pseudocompact. This provides a positive answer to a question posed in \cite{gt} in any model of $ZFC+CH$.
\end{abstract}

\maketitle

\section{Introduction}

In this article, every space will be Tychonoff and every topological group will be Hausdorff (hence, they will be also Tychonoff). For an infinite set $X$, $[X]^{< \omega}$ will denote the family of all finite subsets of $X$ and $[X]^{\omega}$ will denote the family of all countable infinite subsets of $X$.
 A finite set $\{x_0,...., x_l\}$ of elements of an Abelian group $G$  is called {\it independent} if does not contain $0$ and if $\sum_{i \leq l}n_ix_i \neq 0$, then $n_ix_i = 0$ for each $i \leq l$. A nonempty subset $X$ of  $G$ is called {\it independent} if every finite subset of $X$ is independent. The continuum will be denoted by $\co$.
For $A \subseteq \co$, the symbol $\chi_A: \co \to \{0, 1\}$ stands for the characteristic function of $A$. For each $x \in \{0, 1\}^\co$, we define $sup(x) = \{ \xi < \co : x(\xi) \neq 0 \}$.

\medskip

The following generalization of countable compactness was introduced in  \cite{sy}  with the name ``strong pseudocompactness" (the name was changed it in the paper \cite{doras} since the authors noticed that the term ``strong pseudocompactness"  was already used to name a different topological property). This paper \cite{doras} also contains several results on selectively pseudocompact spaces and its relations with some other pseudocompact like properties.

\begin{definition} A space $X$ is called {\it selectively  pseudocompact} if for each sequence $(U_{n})_{n\in \mathbb{N}}$ of pairwise disjoint nonempty open subsets of $X$ there is a sequence $(x_{n})_{n\in \mathbb{N}}$ of points in $X$ and $x \in X$ such that  $x_{n}\in U_{n}$,  for every $n < \omega$, and $x$ is an accumulation point of $\{ x_n : n < \omega\}$.
\end{definition}

It is evident that every countably compact space is selectively  pseudocompact and every selectively  pseudocompact space is pseudocompact. There is a selectively  pseudocompact group that is not countably compact and a pseudocompact group that is not selectively  pseudocompact (for these two examples the reader is referred to \cite{gt}). However, it was pointed out in  \cite{sy} that  pseudocompact groups have a property very similar to selectively  pseudocompactness:

\begin{theorem}\label{chara} {\bf \cite{sy}} For a topological group $G$, the following conditions are equivalent.
\begin{enumerate}
\item $G$ is pseudocompact.

\item For each family $\{ U_{n} : n\in \mathbb{N} \}$ of pairwise disjoint nonempty open subsets of $G$ there is a discrete subset $D$ of $G$ contained in $\bigcup_{n \in \mathbb{N}}U_n $ such that $cl_G(D) \setminus \big(\bigcup_{n \in \mathbb{N}}U_n \big) \neq \emptyset$ and $|D \cap U_{n}| < \omega$, for every $n \in \mathbb{N}$.

\item For each  family $\{ U_{n} : n\in \mathbb{N} \}$ of pairwise disjoint nonempty open subsets of $G$ there is a discrete subset $D$ of $G$ contained in $\bigcup_{n \in \mathbb{N}}U_n $ such that $cl_G(D) \setminus \big(\bigcup_{n \in \mathbb{N}}U_n \big) \neq \emptyset$, $|D \cap U_{n}| \leq \omega$, for every $n \in \mathbb{N}$, and $D \cap U_{n} \neq \emptyset$ for infinitely many $n's$.
\end{enumerate}
\end{theorem}

 A coun\-ta\-bly compact space whose square is not pseudocompact (for one of these spaces see \cite{gj}) is an example of a selectively  pseudocompact space whose square is not pseudocompact. In the realm of topological groups, it is well-known that  ``the product of pseudocompact groups is pseudocompact'', this fact was established by  W. W. Comfort and K. A. Ross in their classical paper \cite{cr}. All these remarks suggest naturally   the following question  listed in \cite{gt} and in \cite{sy2}.

\begin{question}\label{st} Is selectively  pseudocompactness productive in the class of topological groups?
\end{question}

Let us make some comments about a possible solution to Question \ref{st} inside of a model of $ZFC$. It is an old open problem posed by W. W. Comfort whether or not the product of two countably compact groups is countably compact.
In 1980, E. K. van Douwen \cite{vD80} assuming the existence of a countably compact Boolean group without non-trivial convergent sequences constructed, in $ZFC$, two countably compact groups whose product is not countably compact. Thus Comfort's problem reduces to the existence of  a countably compact Boolean group without non-trivial convergent sequences. The first such a topological group was constructed, assuming $CH$, by
A. Hajnal and I. Juh\'asz \cite{HaJu76}. Other known examples of such  topological groups are constructed either by using  some set-theoretic axiom compatible with $ZFC$ (see for instance \cite{ds}, \cite{DiTk03}, \cite{HvM91}, \cite{To99} and \cite{To05ta}), by a\-ssu\-ming the existence of selective ultrafilters on $\omega$ (\cite{GaToWa05}) or in a Random model (see \cite{SzTo09}). The existence of a countably compact group without non-trivial convergent sequences is still unknown in $ZFC$.

 \medskip

 Our purpose in this paper is the construction  for every positive integer $k > 0$, under the assumption of $CH$,  of a topological group  whose $k$-th power is countably compact but its $(k+1)$-th  power  is not selectively  pseudocompact (concerning spaces, Z. Frol\'{\i}k (see \cite{frolik}) constructed for each $1 < k < \omega$, a space $X$ such that  $X^k$ is countably compact and $X^{k+1}$ is not pseudocompact.).  Unfortunately, we could not answer Question \ref{st} by assuming only the axioms of $ZFC$.

\medskip

For the  construction of the required countably compact groups we shall follow some basic ideas from \cite{SzTo14}. In particular, the following notions and proposition   that were used in that paper.

\begin{definition} Let $\lambda$ be an infinite cardinal and let  $Y$ be an infinite subset of a topological product
$X=\prod_{\alpha <\lambda}X_\alpha$.
\begin{enumerate}
 \item $Y$ is called \it{finally dense} in $X$ if there exists $\beta <\lambda$ such that
$\pi_{\lambda \setminus \beta}[Y]$ is dense in $\prod_{\beta \leq \alpha <\lambda}X_\alpha$, where $\pi_{\lambda \setminus \beta}$ is the projection from $X$ onto $\prod_{\beta \leq \alpha <\lambda}X_\alpha$.

\item If every infinite subset of $Y$ is finally dense in $X$, then we say that  $Y$ is  \it{hereditarily finally dense} ($HFD$) in $X$.

\item  $Y$ has property $(P)$ if the projection $\pi_I: \prod_{\alpha <\lambda}X_\alpha \to \prod_{\alpha \in I}X_\alpha$ satisfies that $\pi_I[Y] = \prod_{\alpha \in I}X_\alpha$ for all $I \in [\lambda]^\omega$.
\end{enumerate}
\end{definition}

Property $(P)$ was first considered in \cite{HaJu76} and $HFD$-spaces were introduced in \cite{hj74}. It was proved in \cite{HaJu76} that a $HFD$-subspace of $\{0, 1\}^{\omega_1}$ with property $(P)$ is countably compact. Because of the diagonal, the powers of a $HFD$ cannot be $HFD$ space.

\begin{proposition} {\bf \cite[Prop. 3.2]{DiTk03}} \label{proposition.dt} Let $\lambda$ be an uncountable regular cardinal and let $X=\prod_{\alpha <\lambda}X_\alpha$ be a product of compact metrizable spaces $X_\alpha$ each of which contains at least two points. Suppose that $Y$ is a subset of $X$ such that $\pi_{[0,\beta)}[Y] =\prod_{\alpha <\beta}X_\alpha$ for each $\beta <\lambda$. If $S\subseteq Y$ is an $HFD$ set in $X$, then $S$ has a cluster point in $Y$, but no sequence in $S$ converges. In particular, if $Y$ is $HFD$ in $X$, then it is a countably compact dense subspace of $X$ which does not contain non-trivial convergent sequences.
\end{proposition}

The following result in combination with the Proposition above has been used in \cite{SzTo14} to make finite powers countably compact.

\begin{lemma} \label{menosuno} Suppose that $G$ is the group generated by $\{x_\mu :\, \mu \in {\mathfrak c}\} \subseteq 2^{\omega_1}$ and $\{ \{ \sum_{\mu \in h(i,n)}x_\mu :\, i \in m\}:\, n \in C\}$ has an accumulation point in $G^m$ for each $h:\, m \times C \longrightarrow [{\mathfrak c}]^{<\omega}$, where $1\leq m \leq k$
and $C \in [\omega]^\omega$,  satisfying that the set $\{ h(i,n):\, i<m \text{ and } n\in C\}$ is linearly independent. Then $G^k$ is countably compact.
\end{lemma}

 Given a positive integer $k$. we shall prove, under the assumption of $CH$,  the existence of an independent  set $\{x_\xi:\, \xi <{\mathfrak c}\}\subseteq \{0, 1\}^{\co}$ so that the group generated by this set, $G=\langle \{ x_\xi :\, \xi < {\mathfrak c}\} \rangle$ is $HFD$ with property $(P)$, $G^k$ is countably compact and $G^{k+1}$ is not selectively pseudocompact.  In the second section, we list the conditions that the  generators must have and, based on these conditions, we shall give the details of the topological properties of the group  $G$. In the last section, we proceed to the technical constructions which allows to define such generators with the required properties.

\section{The topological group}

The family of open sets that will guarantee  the destruction of  selectively  pseudocompactness of our topological group will be the one constructed in the next lemma, which is an extension of Lemma 2.1 from \cite{gt}.

\begin{lemma}\label{open} Let $G$ be an non-discrete topological group of order $2$ with identity $0$ and let $k < \omega$ be positive. Then there is a countable family $\{U_n : \, n < \omega\}$ of nonempty open subsets of  $G^{k+1}$ such that if $x_n \in U_n$, for an arbitrary $n < \omega$, then  the set $\{ x_n(i) : i \leq k \ \text{and} \ n < \omega \}$ is a linearly independent set in $G$.
\end{lemma}

\begin{proof} By Lemma 2.1 of \cite{gt}, we can find a sequence of nonempty open sets $\{W_n : \, n < \omega\}$ of $G$ such that $0 \notin \sum_{n\in F}W_n$ for each nonempty finite subset $F$ of $\omega$. Then, for each $n < \omega$, we define $U_n = \prod_{i\leq k} W_{n(k+1)+i}$. For every $n < \omega$, choose $x_n \in U_n$.  By the definition and the properties of the family $\{W_n : \, n < \omega\}$, it is easy to see that  $\{ x_n(i) : i \leq k \ \text{and} \ n < \omega \}$ is  linearly independent in $G$.
\end{proof}

For each $\xi <{\mathfrak c}$, the coordinates of  $x_\xi \in \{0, 1\}^{\co}$ will be  defined inductively.
To do that we shall need to consider the set $[\co]^{<\omega}$ as a vector space over the field $\{0, 1\}$ with the symmetric difference as its group operation and,
for each $\alpha < \co$, we shall construct (see Theorem \ref{main}) suitable  homomorphisms $\psi_\alpha: [\co]^{<\omega} \to \{0,1\}$ which shall determine the coordinates of the $x_\xi$'s by defining $x_\xi(\alpha) = \psi_{\alpha}(\{\xi\})$ for each $\alpha, \xi < \co$. The following  easy lemma will be fundamental to guarantee the linear independence of the generators.

\begin{lemma}\label{cero} Let $\{x_\xi : \xi < \co \}$ be an independent subset of $\{0, 1\}^\co$ and let $\{F_n:\, n < \omega \} \subseteq [\co]^{<\omega} \setminus \{\emptyset\}$. Then the set $\{ \sum_{\xi \in F_n}x_\xi : n < \omega \}$ is independent iff $\{F_n:\, n < \omega \}$ is a linearly independent subset of the vector space $[{\mathfrak c}]^{<\omega}$ over the field $\{0, 1\}$.
\end{lemma}

First of all, we  fix some suitable enumerations:

\smallskip

Put  $I= [\omega +\omega, \co)$.

\noindent $\bullet$ For each $\xi \in I$, fix $t_\xi \in \{0, 1\}^{\xi}$ so that
for every $\alpha <{\co}$ and for every $x \in \{0, 1\}^\alpha$ there exists $\xi \in I$ such that $\xi > \alpha$ and $t_\xi |_{\alpha} =x$.

\noindent $\bullet$  Let $\{ h_\xi :\, \xi \in I\}$ be an enumeration of all  functions $h:\, m \times C \longrightarrow [{\mathfrak c}]^{<\omega}$, where $1\leq m \leq k$
and $C \in [\omega]^\omega$,  satisfying that the set $\{ h(i,n):\, i<m \text{ and } n\in C\}$ is linearly independent, and $\bigcup_{(i,n)\in m_\xi \times C_\xi}h_\xi (i,n) \subseteq \xi$ for each $\xi \in I$.  Put  $dom(h_\xi) = m_\xi \times C_\xi$ for each $\xi \in I$.

\noindent $\bullet$ Let $\{(f_\xi,\vec{F}_\xi):\, \xi \in I\}$ and $\vec{F}_\xi = (F^0_\xi, \ldots , F_\xi^k)$ be an enumeration of  all pairs $(f,\vec{F})$, where $\vec{F}=(F^0, \ldots , F^k) \in ([{\mathfrak c}]^{<\omega})^{k+1}$ and  $f:\,  \omega \longrightarrow [{\mathfrak c}]^{<\omega}$ is a function such that $\{ f(n):\,  n < \omega \}$ is linearly independent,   the dimension of the quotient of $[\xi]^{<\omega}$ by the subgroup $\langle\{ f_\xi(n):\,  n < \omega \}\rangle$ is infinite, and $\big(\bigcup_{n < \omega} f_\xi(n) \big) \cup \big(\bigcup_{i\leq k} F^i_\xi\big) \subseteq \xi$.

\medskip

To carry out the construction, let us  state the conditions that the  homomorphisms $\psi_\alpha$'s should have when  $\alpha  \in I$:

\medskip

$a)$ $\psi_\alpha (F^0_\alpha)=1$ if $F^0_\alpha \neq \emptyset$ for each $\alpha \in I$;

$b)$ $\psi_\alpha (\{\xi\})=t_\xi (\alpha)$ for each $\alpha \in I$ and  $\xi \in I$ with $\alpha < \xi$; and

$c)$ $\{ n < \omega :\, \forall i \leq  k \big(\psi_\alpha(F^i_\alpha ) = \psi_\alpha [f_\alpha(n(k+1)+i)] \big) \}$ is finite
for each $\alpha \in I$.

\noindent  For every $\gamma <\beta \in I$, we shall consider the  homomorphism $\Psi |_{[\gamma, \beta)}: [{\mathfrak c}]^{<\omega} \longrightarrow \{0, 1\}^{[\gamma, \beta)}$ defined by $\Psi |_{[\gamma , \beta)}(F) = \big(\psi_\xi(F)\big)_{\gamma \leq \xi < \beta} $ for each $F \in  [{\mathfrak c}]^{<\omega}$. Then, we have that:

$d)$ $\{ (\Psi_{[\gamma , \beta)}[h_\gamma(0,n)],..., \Psi_{[\gamma , \beta)}[h_\gamma(m_{\gamma}-1,n)]) :\, n \in C_\gamma\}$ is dense in $\big(\{0, 1\}^{[\gamma , \beta )}\big)^{m_\gamma}$ for each $\gamma <\beta  \in I$.

\noindent Let us see in the next theorem that all these four conditions will give us the desired topological group.

\begin{theorem}\label{meromero}{\bf [CH]}   Fix a positive integer $k$ and consider the   homomorphisms $\{ \psi_\alpha: [\co]^{<\omega} \to \{0,1\} :  \alpha < \co\}$  satisfying the conditions $a)-d)$ from above.  For each $\xi < \co$, we define  $x_\xi \in \{0, 1\}^\co$ by  $x_\xi(\alpha) = \psi_{\alpha}(\{\xi\})$ for each $\alpha < \co$. $G=\langle \{ x_\xi :\, \xi < {\mathfrak c}\} \rangle$ is $HFD$ with property $(P)$, $G^k$ is countably compact and $G^{k+1}$ is not selectively pseudocompact.
\end{theorem}

\begin{proof} Lemma \ref{cero}  and  condition  $a)$ imply  directly that the generators $\{ x_\xi : \xi < \co \}$ are linearly independent.
Fix  $0 < m \leq k$.  Next, we shall show that $\{ ( F(i,n))_{i<m}:\, n \in C\}$ has an accumulation point in the topological group  $G^m$ whenever $\{ F(i,n):\, i < m \wedge n \in C \}$ is a linearly independent of $G$.

Then, choose the function $h: m \times C \to [\co]^{< \omega}$ defined by $F(i,n)=\sum_{\mu \in h(i,n)} x_\mu$, for each  $(i,n) \in m \times C$. This $h$ is equal to   $h_\gamma$ for some $\gamma \in I$. Observe that
$$
\Psi |_{[\gamma , \beta)}(F) = \big(\psi_\xi(F)\big)_{\gamma \leq \xi < \beta} = \pi_{[\gamma , \beta)}(\sum_{\xi \in F}x_\xi)\footnote{Here, $ \pi_{[\gamma , \beta)}: \{0, 1\}^{\co} \to \{0,1\}^{[\gamma , \beta)}$ is the projection map.},
$$
 for every $F \in [\co]^{< \omega}$ and  $\gamma \leq \beta < \co$. So, we have that
$$
(\sum_{\xi \in F_n^0}x^0_{\xi},\ldots,\sum_{\xi \in F_n^m}x^m_{\xi}) =  (\Psi_{[\gamma , \beta)}[h_\gamma(0,n)],..., \Psi_{[\gamma , \beta)}[h_\gamma(m-1,n)])
$$
for every $ n \in C$. According to clause $d)$, we know that
$\{ (\sum_{\xi \in F_n^0}x^0_{\xi},\ldots,\sum_{\xi \in F_n^m}x^m_{\xi}) :\, n \in C\}$
is dense in $\big(\{0, 1\}^{[\gamma , \beta )}\big)^{m}$ for every $\gamma <\beta < \alpha$. Thus, we obtain that this sequence is finally dense in $(\{0,1\}^\co)^k$. Now, all its infinite subsets are also linearly independent, hence the sequence is $HFD$.
To prove that  $G^m$ has property  $(P)$ in $(\{0,1\}^{\omega_1})^m$ consider the projection map $\pi_\alpha: (\{0,1\}^{\omega_1})^m \to (\{0,1\}^\alpha)^m$ for an arbitrary  $\omega+\omega \leq \alpha < \omega_1$. Fix $y  = (y_0,\ldots,y_m) \in (\{0,1\}^\alpha)^m$. Then, for every $i \leq m$  we can find  $\xi_i \in I$  such
that $\xi_i > \alpha$ and $t_{\xi_i} |_{\alpha} = y_i$. By condition $b)$, we have that $t_{\xi_i} (\beta) = \psi_\beta (\{\xi_i\})= x_{\xi_i}(\beta) = y_i(\beta)$ for each $i \leq m$ and for each $\beta < \alpha$. This means that $\pi_\alpha(\sum_{i\leq m}x_{\xi_i}) = y$. Thus, we have proved that $G^m$ has property  $(P)$.
 By Proposition  \ref{proposition.dt}, we conclude that this sequence has an accumulation point in $G^m$. By Lemma \ref{menosuno}, it follows that $G^k$ is countably compact. Therefore, $G^m$  is strongly  pseudocompact for all positive integer $m \leq k$.

\medskip

It remains to  show that $G^{k+1}$ cannot be selectively  pseudocompact.  Choose a countable family $\{U_n : \, n < \omega\}$  consisting of  open subsets  of  $G^{k+1}$ given by Lemma  \ref{open}. For each  $n < \omega$ fix $y_n \in U_n $. Then the set $\{ y_n(i) : i \leq k \ \text{and} \ n < \omega \}$ is  linearly independent in $G$.
Fix $x = (\sum_{\xi \in F^0}x_\xi, \ldots , \sum_{\xi \in F^k}x_\xi) \in G^{k+1}$.  Define $f: \omega \to [\co]^{<\omega}$ so that $y_n(i) = \sum_{\xi \in f(n(k+1) +i)}x_\xi $ for each $i \leq k$ and $n \in \mathbb{N}$.  Now, choose $\beta \in I$ for which $(f,(F^0, \ldots , F^k)) = \big(f_\beta,(F_\beta^0, \ldots , F_\beta^k)\big)$.  By clause $c)$,  for every $i \leq k$ we know that
$$
\pi_\beta(\sum_{\xi \in F^i}x_\xi) = \sum_{\xi \in F^i}x_\xi(\beta) = \sum_{\xi \in F^i}\psi_\beta(\{\xi\}) =  \psi_\beta(F^i)  = \psi_\beta [f(n(k+1)+i)] = y_n(i)(\beta) = \pi_\beta\big(y_n(i)\big)
$$
holds for finitely many $n \in \mathbb{N}$.
It then follows that $\{ \{ y_{n(k+1)+i}:\, \leq k \}:\, n< \omega\}$ does not have an accumulation point in $G^{k+1}$.
Therefore, the family of open sets $\{ \prod_{i<K+1} U_{(k+1)n+i}:\, n< \omega\}$ witnesses that $G^{K+1}$ is not selectively  pseudocompact.
\end{proof}

\section{The construction}

In this section, we shall construct the homomorphisms $ \psi_\alpha$'s that will allow us to define the generators of the desired topological group.

\medskip

{\bf Notation}:  In what follows, for a positive $k < \omega$,  the italic letter ${\mathfrak B}$ will stand for the $\co$-sized set of functions  of the form $f:\, m_f \times C_f \longrightarrow [\co]^{<\omega}$, where  $1\leq m_f \leq k$ and $C_f \in [\omega]^\omega$, such that the set $\{f(i,n):\, i < m_f \ \text{ and } \ n \in C_f \}$ is  linearly independent.

\begin{lemma} \label{Lemma1} Let $\omega \leq \alpha < \co$ and $k$ a positive integer. If $\{ F^i:\, i \leq k \} \subseteq [\alpha]^{<\omega}$ and  ${\mathcal B}$  is a countable subset of ${\mathfrak B}$ such that $\bigcup_{(i,n) \in m_f  \times C_f}f(i,n) \subseteq \alpha$ for every $f \in \mathcal{B}$, then there exists  a homeomorphism $\psi:\, [\alpha]^{<\omega} \longrightarrow \{0, 1\}$ such that:\,
\begin{enumerate}
\item[i)] $\psi (F^0)=1$ if $F^0\neq \emptyset$;

\item[ii)] $ \{ n < \omega : \, \forall i \leq k \big(\psi(F^i)=\psi( \{ n(k+1)+i\})\big) \}$ is finite; and

\item[iii)] $\{ n < \omega:\, \forall i< m_f\big(\psi (f(i,n))=\sigma (i) \big)\}$ is infinite  for every $f \in {\mathcal B}$ and  $\sigma \in \{0,1\}^{m_f}$.
\end{enumerate}
\end{lemma}

\begin{proof} Enumerate the set of all  pairs $(f,\sigma)$ such that $ f\in {\mathcal B}$ and $ \sigma \in \{0,1\}^{ m_f}$
as $\{(f_n,\sigma_n):\, n < \omega \}$ so that each $(f,\sigma)$ appears infinitely often. Choose a positive integer $N_0$ for which
$\omega \cap (F^0 \cup \ldots \cup F^k)\subseteq N_0(k+1)$. Set $E_0:= F^0 \cup \ldots \cup F^ k \cup N_0(k+1)$ and
fix an homomorphism $\phi_0 :\, [E_0]^{<\omega} \longrightarrow \{0, 1\}$   so that $\phi_0(F^0)=1$ whenever $F^0\neq \emptyset$.
Suppose that for a positive integer $l$  we have defined the sets $\{N_s:\, s \leq l\}$, $\{n_s:\, s<l\}$, $\{E_s:\, s\leq l \}$ and $\{\phi_s:\, s \leq l\}$
so that the following hols:
\begin{enumerate}
\item[$a)$] $\{N_s:\, s\leq l\}$ is a strictly increasing sequence of positive integers;

\item[$b)$] $\{n_s:\, s <l\}$ is a strictly increasing sequence of positive integers;

\item[$c)$] $E_s:= F^0 \cup \ldots \cup F^k \cup N_s(k+1) \cup \big(\bigcup\{ f_r(i,n_r):\, r<s \text{ and } i <m_r \}\big)$;

\item[$d)$] for each $ s\leq l$, $\phi_s :\, [E_s]^{< \omega} \longrightarrow \{0, 1\}$ is a homomorphism that  extends $\phi_r$ for every $r<s$;

\item[$e)$] if $N_s \leq  n < N_{s+1}$, then there exists $i <k+1$ such that $ \phi_s(F^i)\neq \phi_s(\{n(k+1)+i\})$ for each $s < l$;

\item[$f)$] $( \bigcup \{ f_r(i,n_r):\, r<s \text{ and } i< m_{f_r} \}) \cap \omega \subseteq N_s(k+1)$ for each $s\leq l$; and

\item[$g)$] $\phi_s(f_s(i,n_s))=\sigma_s(i)$ for each $s<l$ and $i<m_{f_s}$.
\end{enumerate}
We will construct $n_l$, $N_{l+1}$, $E_{l+1}$ and $\phi_{l+1}$.

\smallskip

{\bf Claim.} For every $E \in [\co]^{< \omega}$ there is  $n < \omega$ such that the set $\{ \{\beta\}:\, \beta \in E \} \cup \{f_l(i,n):\, i <m_{f_l}\}$ is linearly independent.

\smallskip

{\it Proof of the Claim:} In fact, suppose that this is not the case. Then for each $n< \omega$ there exists a non-empty set $F_n$ of $m_{f_l}$ such that $\sum_{k \in F_n}f_l(k,n)\in [E]^{<\omega}$. As the group $[E]^{<\omega}$ is finite, it follows that there exist two distinct integers $m_0$ and $m_1$ such that $\sum_{i \in F_{m_0}}f_l(i,m_0)=\sum_{i \in F_{m_1}}f_l(i,m_1)$, but this contradicts the assumption that $\{f_l(i,n):\, i<m_{f_l} \text{ and } n < \omega\}$ is linearly independent.

\medskip

By the Claim, we can find $n_l < \omega$ such that $n_{l-1} < n_l$ and  $\{ \{\beta\}:\, \beta \in E_l \} \cup \{f_l(i,n_l):\, i <m_{f_l}\}$ is linearly independent. Choose a positive integer $N_{l+1}$ so that  $N_{l+1} > N_l$ and $\omega \cap \big(\bigcup_{i<m_{f_l}}f_l(i,n_l)\big) \subseteq N_{l+1}(k+1)$.
By condition $f)$ for $l$, we obtain  that condition $f)$ for $l+1$ is satisfied. Define
$$
E_{l+1}:= F^0 \cup \ldots \cup F^k \cup N_{l+1}(k+1) \cup \big(\bigcup\{ f_s(i,n_s):\, s<l+1 \text{ and } i <m_s \}\big).
$$
It is  straightforward to see that conditions $a)$, $b)$ and $c)$ are satisfied for $l+1$.

\medskip

Recall that the set $\{ \{\beta\}:\, \beta \in E_l\} \cup \{f_l(i,n_l):\, i <m_{f_l}\}$ was chosen to be independent and the set $\{ \{\beta\}:\, \beta \in E_l\}\cup \{ \{N_l(k+1) +i\}:\, i< k+1 \}$ is independent because of clause $f)$ for $l$. Since $|\bigcup_{i<m_{f_l}}f_l(i,n_l)|< |\{ \{N_l(k+1) +i\}:\, i<k+1 \}|$, it follows that there exists $i_{N_l}$ for which the set
$$
\{ \{\beta\}:\, \beta \in E_l\} \cup \{f_l(i,n_l):\, i <m_{f_l}\}\cup \{\{ N_l(k+1)+i_{N_l}\}\}
$$
is linearly independent. Thus, we may proceed inductively to find
$i_n<k+1$, for each $n\in [N_l, N_{l+1})$, such that  $\{ \{\beta\}:\, \beta \in E_l\} \cup \{f_l(i,n_l):\, i <m_{f_l}\}\cup \{\{ n(k+1)+i_{n}\}:\, N_l\leq n <N_{l+1}\}$ is linearly independent. Hence, it is possible to define a homomorphism  $\phi_{l+1}:\, [E_{l+1}]^{<\omega} \longrightarrow \{0, 1\}$ that satisfies the following:
\begin{enumerate}
\item[$A)$] $\phi_{l+1}(\{\beta\})=\phi_l(\{\beta\})$ for each $\beta \in E_l$;

\item[$B)$] $\phi_{l+1}(\{n(k+1)+i_n\})=1-\phi_{l+1}(F^{i_n})$ for each $N_l\leq n <N_{l+1}$ and

\item[$C)$] $\phi_{l+1} (f(i,n_l))=\sigma_l(i)$ for each $i<m_{f_l}$.
\end{enumerate}
It follows from condition $A)$ that $\phi_{l+1}$ extends $\phi_l$ and so condition $d)$ holds. Condition  $e)$ follows directly from $B)$ and condition $C)$ implies $g)$.

\medskip

Let $\psi$ be an extension of $\bigcup_{l<\omega}\phi_l$. Condition $i)$ follows from the fact that $\psi$ extends $\phi_0$.
 According to $e)$, if $N_{l} \leq n$, then there exists $i <k+1$ such that $ \phi_l(F^i)\neq \phi_s(\{n(k+1)+i\})$. This clearly implies that the set
$ \{ n < \omega : \, \forall i \leq k \big(\psi(F^i)=\psi( \{ n(k+1)+i\})\big) \}$ is finite  and so condition $ii)$ is satisfied.
Let  $f \in {\mathcal B}$ and  $\sigma \in \{0,1\}^{m_f}$. Choose $s < \omega$ so that $f_s = f$ and $\sigma_s = \sigma$. By condition $g)$, we have that
 $\phi_s(f_s(i,n_s))=\sigma_s(i)$ for each $i<m_{f_s}$ and since the set  $\{n < \omega :\, (f_s,\sigma_s)=(f_n,\sigma_n)\}$ is infinite, we conclude that
 $\{n < \omega :\, \forall i< m_f \big(\psi(f(i,n))=\sigma(i)\big) \}$  is also infinite.
Then, condition $iii)$ holds.
\end{proof}

\begin{lemma} \label{Lemma2} Suppose that  $k$ is a positive integer and $\omega+\omega \leq \alpha <\omega_1$. Let $F^0, \ldots F^k \in [\alpha]^{<\omega}$ and let ${\mathcal B}$ be a countable subset of ${\mathfrak B}$ such that $\bigcup_{(i,n) \in m_f  \times C_f}f(i,n) \subseteq \alpha$ for each $f \in \mathcal{B}$. If $\{H_n:\, n < \omega\}$ is a linearly independent subset of $[\alpha]^{<\omega}$ such that the quotient $[\alpha]^{<\omega}$ by the group generated by  $\{H_n:\, n < \omega\}$ is infinite, then there exists a homomorphism $\psi:\, [\alpha]^{<\omega} \longrightarrow \{0, 1\}$ such that
\begin{enumerate}
\item[I)] $\psi(F^0)=1$ whenever $F^0 \neq \emptyset$;

\item[II)]  $\{ n< \omega :\,  \forall i \leq k\big(\psi (F^i)=\psi (H_{n(k+1)+i})\big) \}$ is finite; and

\item[III)]  $\{n \in C_f:\, \forall i < m_{f} \big(\psi(f(i,n))=\sigma(i) \big) \}$ is infinite for every $f \in {\mathcal B}$ and $\sigma \in \{0,1\}^{m_f}$.
\end{enumerate}
\end{lemma}

\begin{proof} By the hypothesis, there exists $\{H_\xi:\, \omega \leq \xi < \alpha\}$ such that $\{H_n:\, n< \omega \}\cup \{H_\xi:\, \omega \leq \xi < \alpha\}$
is a basis for $[\alpha ]^{<\omega}$. Consider the homomorphism $\Phi:[\alpha]^{<\omega} \to[\alpha]^{<\omega}$ defined by $\Phi(H_\xi)=\{\xi\}$ for each $\xi < \alpha$.

For each  $f \in {\mathcal B}$ choose $h_f:\, m_{h_f} \times C_{h_f} \longrightarrow [\alpha]^{<\omega}$ so that $h_f(i,n)=\Phi(f(i,n))$ for all $i<m_{f}$ and $n\in C_{f}$.
Put $\mathcal{A} = \{ h_f:\, f \in {\mathcal B}\}$. Notice that $\{h_f(i,n):\, i <m_f \text{ and } n \in C_f\}$ is linearly independent for every $f \in \mathcal{B}$.
for every $f \in \mathcal{B}$. Now, we apply Lemma \ref{Lemma1} to obtain a homomorphism $\tilde{\psi}: [\alpha]^{<\omega} \to [\alpha]^{<\omega}$ such that:
\begin{enumerate}
\item[$i)$] $\tilde{\psi}(\Phi(F^0))=1 $ if $\Phi(F^0)\neq \emptyset$;

\item[$ii)$] $\{ n < \omega:\, \forall i<k+1 \big(\tilde{\psi}(\Phi(F^i))=\tilde{\psi}(\{n(k+1)+i\}) \big)\}$ is finite; and

\item[$iii)$] $\{n \in C_{f}:\,  \forall i<m_{f} \big(\tilde{\psi}(h_f(i,n))=\sigma(i) \big)\}$ is infinite for each $f \in {\mathcal B}$ and  $\sigma \in \{0,1\}^{m_f}$.
\end{enumerate}
Define $\psi= \tilde{\psi}\circ\Phi$. From clause $i)$ it follows that $\psi(F^0)=1 $ if $\Phi(F^0)\neq \emptyset$ which is equivalent to $F^0\neq \emptyset$ since $\Phi$ is an isomorphism. Thus, $I)$ holds. According to condition $ii)$ and the fact $\Phi(H_n)=\{n\}$ for every $n < \omega$, we have  that $ \{ n < \omega:\,  \forall i<k+1 \big(\psi(F^i)=\psi(H_{n(k+1)+i})\big)\}$ is finite. Thus, $II)$ is satisfied. By condition $iii)$ and the definitions of ${\mathcal B}$ and $\Phi$, we obtain  that
$\{n \in C_f:\,  \forall i<m_f\big(\psi(f(i,n))=\sigma(i)\big)\}$ is infinite for each $f \in {\mathcal B}$ and $\sigma \in \{0,1\}^{m_f}$, which is condition $III)$.
\end{proof}

Now, we turn out to the construction of the required homeomorphisms in the next theorem. To explain the properties of such homomorphisms, it is necessary to use
 the enumerations from the  second section, which will be listed again for the reader's convenience:

\smallskip

\noindent $\bullet$ Set $I= [\omega +\omega, \omega_1)$.

\noindent $\bullet$ For each $\xi \in I$, fix $t_\xi \in \{0, 1\}^{\xi}$ so that
for each $\alpha <{\omega_1}$ and for each $x \in \{0, 1\}^\alpha$ there exists $\xi \in I$ such that $\xi > \alpha$ and $t_\xi |_{\alpha} =x$.

\noindent $\bullet$  Let $\{ h_\xi :\, \xi \in I\}$ be an enumeration of all  functions $h:\, m \times C \longrightarrow [{\mathfrak c}]^{<\omega}$, where $1\leq m \leq k$
and $C \in [\omega]^\omega$,  satisfying that the set $\{ h(i,n):\, i<m \text{ and } n\in C\}$ is linearly independent, and $\bigcup_{(i,n)\in m_\xi \times C_\xi}h_\xi (i,n) \subseteq \xi$ for each $\xi \in I$.  Put  $dom(h_\xi) = m_\xi \times C_\xi$ for each $\xi \in I$.

\noindent $\bullet$ Let $\{(f_\xi,\vec{F}_\xi):\, \xi \in I\}$ be an enumeration of  all pairs $(f,\vec{F})$, where $\vec{F}=(F^0, \ldots , F^k) \in [[{\mathfrak c}]^{<\omega}]^{k+1}$ and  $f:\,  \omega \longrightarrow [{\mathfrak c}]^{<\omega}$ is a function such that $\{ f(n):\,  n < \omega \}$ is linearly independent,   the dimension of the quotient of $[\xi]^{<\omega}$ by the subgroup $\langle\{ f(n):\,  n < \omega \}\rangle$ is infinite, and $\big(\bigcup_{n < \omega} f(n) \big) \cup \big(\bigcup_{i\leq k} F^i_\xi\big) \subseteq \xi$.

\begin{theorem}\label{main}{\bf [CH]}  For every $\alpha < \co$, there is a homomorphism $\psi_\alpha:\, [{\mathfrak c}]^{<\omega} \longrightarrow \{0, 1\}$  such that:

$a)$ $\psi_\alpha (F^0_\alpha)=1$ if $F^0_\alpha \neq \emptyset$ for each $\alpha \in I$;

$b)$ $\psi_\alpha (\{\xi\})=t_\xi (\alpha)$ for each $\alpha \in I$ and  $\xi \in I$ with $\alpha < \xi$; and

$c)$ $\{ n < \omega :\, \forall i \leq  k \big(\psi_\alpha(F^i_\alpha ) = \psi_\alpha [f_\alpha(n(k+1)+i)] \big) \}$ is finite
for each $\alpha \in I$.

\noindent  For every $\gamma <\beta \in I$, we let $\Psi |_{[\gamma, \beta)}: [{\mathfrak c}]^{<\omega} \longrightarrow \{0, 1\}^{[\gamma, \beta)}$ be the   homomorphism defined by $\Psi |_{[\gamma , \beta)}(F) = \big(\psi_\xi(F)\big)_{\gamma \leq \xi < \beta} $ for each $F \in  [{\mathfrak c}]^{<\omega}$. Then, we must have that:

$d)$ $\{ (\Psi_{[\gamma , \beta)}[h_\gamma(0,n)],..., \Psi_{[\gamma , \beta)}[h_\gamma(m_{\gamma}-1,n)]) :\, n \in C_\gamma\}$ is dense in $\big(\{0, 1\}^{[\gamma , \beta )}\big)^{m_\gamma}$ for each $\gamma <\beta  \in I$.
\end{theorem}

\begin{proof} For $\alpha <\omega+\omega$, we  let  $\psi_\alpha:\, [{\mathfrak c}]^{<\omega} \longrightarrow \{0, 1\}$ be the homomorphism satisfying $\psi_\alpha (\{\xi\})=0 $ if $\xi < \omega +\omega$ and $\psi_\alpha (\{\xi\})= t_\xi (\alpha)$ for each $\xi \in I$.

\bigskip
 Now suppose that $\omega + \omega \leq \gamma < \omega_1$ and that  we have defined a suitable homeomorphism $\psi_\alpha:\, [{\mathfrak c}]^{<\omega} \longrightarrow \{0, 1\}$ for each $\alpha < \gamma$ so that the  following inductive conditions are satisfied:

$A)$ $\psi_\beta (F^0_\beta)=1$ if $F^0_\beta \neq \emptyset$ for each $\beta \in I \cap \gamma$;

$B)$ $\psi_\beta (\{\xi\})=t_\xi (\beta)$ for each $\beta \in I \cap \gamma$ and for each $\xi \in I$ with $\beta < \xi$;

$C)$ $\{ n < \omega :\, \forall i \leq  k \big(\psi_\beta [f_\beta((k+1)n+i)]= \psi_\beta(F^i_\beta ) \big) \}$ is finite
for each $\beta \in I \cap \gamma$; and

$D)$ $\{ (\Psi_{[\beta , \alpha)}[h_\beta(0,n)],..., \Psi_{[\beta , \alpha)}[h_\beta(m_{\beta}-1,n)]) :\, n \in C_\beta\}$ is dense in $\big(\{0, 1\}^{[\beta , \alpha )}\big)^{m_\beta}$ for each $\beta \in I \cap \gamma$ and for each $\beta <\alpha < \gamma$.

\noindent Now, we proceed to construct the homeomorphism $\psi_\gamma:\, [{\mathfrak c}]^{<\omega} \longrightarrow \{0, 1\}$. By applying  Lemma \ref{Lemma2} to $\gamma$, $\{F^0_\gamma, \ldots , F^k_\gamma\}$, ${\mathcal B} = \{ h_\beta : \beta \leq \gamma \}$ and
 $\{ f_\gamma(n) : n < \omega \}$,  we can find a homomorphism $\psi_\gamma:\, [\gamma]^{<\omega} \longrightarrow \{0, 1\}$ satisfying:

 $i)$ $\psi_\gamma(F^0_\gamma)=1$ if $F^0_\gamma \neq \emptyset$;

 $ii)$ $\{ n< \omega:\, \forall i \leq k \big(\psi_\gamma (f_\gamma (n(k+1)+i) = \psi_\gamma (F_\gamma^i) \big)\}$ is finite and

 $iii)$ $\{n \in C_{\beta}:\, \forall i < m_{h_\beta}\big (\psi_\gamma(h_{\beta}(i,n))=\sigma(i) \big) \}$ is infinite  for each  $\beta \leq \gamma$ and for each   $\sigma \in \{0,1\}^{m_{\beta}}$.

\noindent Next, we extend $\psi_\gamma$ to a homeomorphism $\psi_\gamma :\, [{\mathfrak c}]^{<\omega} \longrightarrow \{0, 1\}$ such that $\psi_{\gamma}(\{\xi\})=t_\xi (\gamma)$ for each $\xi \in I$ with $\gamma <\xi$. Thus, conditions $a)$ and $b)$ are clearly satisfied. It follows directly from clause $D)$ that

\smallskip

\noindent $(*)$ \ \  $\{ (\Psi_{[\beta , \gamma)}[h_\beta(0,n)],..., \Psi_{[\beta , \gamma)}[h_\beta(m_{\beta}-1,n)]) :\, n \in C_\beta\}$ is dense in $\big(\{0, 1\}^{[\beta , \gamma )}\big)^{m_\beta}$ for each $\beta \in I \cap \gamma$.

\smallskip

\noindent By condition $iii)$, we know that $\{ \big(\psi_\gamma(h_{\beta}(0,n)),...., \psi_\gamma(h_{\beta}(m_\beta -1,n))\big) : n \in C_\beta \} = \{0,1\}^{m_{\beta}}$.
Hence,

\smallskip

\noindent $(**)$ \ \  $\{ (\Psi_{[\beta , \gamma]}[h_\beta(0,n)],..., \Psi_{[\beta , \gamma]}[h_\beta(m_{\beta}-1,n)]) :\, n \in C_\beta\}$ is dense in $\big(\{0, 1\}^{[\beta , \gamma ]}\big)^{m_\beta}$ for each $\beta \in I \cap \gamma$.

\smallskip

\noindent Thus, condition $d)$ holds for $\gamma$. Condition $c)$ follows directly from condition $C)$ and $iii)$. This shows the theorem.
\end{proof}



\end{document}